\documentclass[11pt]{article}

\usepackage{amsmath,amssymb,amsthm,mathtools}
\usepackage{enumitem}
\usepackage[expansion=false]{microtype}
\usepackage[authoryear,round]{natbib}
\usepackage{hyperref}
\usepackage{cleveref}
\usepackage{complexity}

\newtheorem{theorem}{Theorem}[section]
\newtheorem{proposition}[theorem]{Proposition}

\newtheorem{corollary}[theorem]{Corollary}

\theoremstyle{definition}
\newtheorem{definition}[theorem]{Definition}
\newtheorem{example}[theorem]{Example}

\theoremstyle{remark}

\newcommand{\Rplus}{\mathbb{R}_{+}}

\newcommand{\vx}{\mathbf{x}}
\newcommand{\vs}{\mathbf{s}}
\newcommand{\simplex}{\mathcal{B}}
\newcommand{\supp}{\operatorname{supp}}
\newcommand{\eff}{r}
\newcommand{\elas}{E}
\newcommand{\argmax}{\operatorname*{arg\,max}}

\newenvironment{keywords}
  {\par\noindent\textbf{Keywords: }}
  {\par}

\title{Optimal Practice Allocation\\
Under Learning Saturation}

\author{Shrisha Rao}

\date{}

\begin{document}

\maketitle

\begin{abstract}
A saying attributed to Bruce Lee unfavorably contrasts a martial
artist who has practiced ten thousand kicks once each with another who
has practiced a single kick ten thousand times.  We read the saying as
implicitly raising a question about how to allocate a fixed practice
budget among several skills, and we show that the answer is governed
by two ingredients: the shape of the learning curve that converts
practice into skill, and the rule by which separate skills are
aggregated into overall effectiveness. Neither ingredient alone
settles the matter.  Our central result reduces the multivariable
allocation problem to the maximization of a single scalar
\emph{efficiency function} \(\eff(x)=f(x)^{p}/x\); under a simple
uniqueness and divisibility condition, every optimal practice schedule
is \emph{balanced}, dividing the budget equally among a definite
number of skills. The optimal number of skills is then determined by a
transparent criterion equating the elasticity of the learning curve to
the reciprocal of the aggregation parameter. Hard-saturation models,
heterogeneous learning rates, a sharp specialization threshold, and a
genuinely intermediate optimal repertoire all follow as consequences.
\end{abstract}

\begin{keywords}
resource allocation, separable optimization, practice allocation,
learning curves, specialization, diminishing returns, elasticity,
majorization
\end{keywords}

\noindent\textbf{2020 Mathematics Subject Classification}. Primary 90C30;
Secondary 90C25, 90C26, 90C27.

\section{Introduction: One Kick or Ten Thousand?}
\label{sec:intro}

As a saying widely attributed to Bruce Lee has it, ``I fear not the man who
has practiced ten thousand kicks once. But I fear the man who has practiced
one kick ten thousand times'' \citep[p.~109]{lee2020bewater}.

Whatever be its sensibility as advice about the martial arts, the
aphorism poses a natural mathematical question.  Consider a learner
who has a finite amount of practice time and several skills on which
to spend it.  Repeated practice improves a skill, but perhaps with
diminishing returns and perhaps only up to a point.  Should the
learner concentrate all available effort on a single skill, spread it
evenly over many skills, or choose some intermediate repertoire?  The
question resonates well beyond martial arts, from the psychology of
deliberate practice \citep{ericsson1993deliberate} to its popular
restatement as a rule of ten thousand hours
\citep{gladwell2008outliers}; our aim is to isolate the mathematics
beneath it.

The question is a resource-allocation problem, but one with two
distinct ingredients.  The first is a collection of \emph{learning
curves}, which describe how practice produces skill; the long-studied
empirical tendency for gains from practice to diminish, often modeled
by concave, power-law curves, records this \citep{newell1981practice}.
The second is an \emph{aggregation rule}, which describes how several
acquired skills combine into overall effectiveness.  Our main theme is
that neither ingredient alone determines whether specialization is
desirable, and that the interaction between them is captured by a
single scalar function.

We begin with the simplest possible framework.  There are \(n\) skills,
indexed by \(1,\ldots,n\), and a total practice budget \(B>0\); throughout
we write \(\Rplus=[0,\infty)\), so that a coordinate may equal zero. An
allocation is a vector
\[
    \vx=(x_1,\ldots,x_n)\in\Rplus^{\,n},
    \qquad
    \sum_{i=1}^n x_i=B,
\]
where \(x_i\) is the practice devoted to skill \(i\). Skill \(i\) has a
learning curve
\[
    f_i:\Rplus\longrightarrow\Rplus,
    \qquad
    f_i(0)=0,
\]
so that the attained skill level is \(f_i(x_i)\); the normalization
\(f_i(0)=0\) records that skill is built only through practice. The
feasible set is the scaled simplex
\[
    \simplex_B
    =
    \Bigl\{
        \vx\in\Rplus^{\,n}:
        \textstyle\sum_{i=1}^n x_i=B
    \Bigr\}.
\]
Given an increasing aggregation function \(A:\Rplus^{\,n}\to\Rplus\), the
general problem is
\begin{equation}
    \max_{\vx\in\simplex_B}
    A\bigl(f_1(x_1),\ldots,f_n(x_n)\bigr).
    \label{eq:general-problem}
\end{equation}

The problem~\eqref{eq:general-problem} is an applied mathematical
model of resource allocation: a fixed budget is distributed among
competing activities whose returns are governed by learning curves and
whose contributions are combined through an aggregation rule. The
Bruce Lee example is a concrete motivating instance, but the results
below concern this broader class of allocation models. The subsequent
analysis focuses on the separable case, with a fixed budget divided
among activities and, in the principal model, a tunable power-mean
aggregation rule.
 
Our principal contribution is a reduction of the multivariable problem
\eqref{eq:general-problem}, under power-mean aggregation, to the
maximization of a single scalar \emph{efficiency function}
\(\eff(x)=f(x)^{p}/x\). \Cref{thm:efficiency-bound} bounds the value of
every feasible allocation by \(B\,\sup_{x}\eff(x)\); \Cref{thm:structural}
shows that, whenever the efficiency optimum divides the budget evenly,
every optimal schedule is \emph{balanced}, dividing the budget equally
among a definite number of skills; and the elasticity criterion of
\Cref{prop:elasticity}, namely \(\elas_{f}(x^{*})=1/p\), fixes that number
by equating the elasticity of the learning curve to the reciprocal of the
aggregation exponent. The hard-saturation model of \Cref{sec:saturation}
and the heterogeneous model of \Cref{sec:heterogeneous}, the latter
including a reduction to the \(0\)--\(1\) knapsack problem
(\Cref{prop:knapsack}), arise as consequences and variants of this
analysis.
 
The separable problem \eqref{eq:general-problem} sits within the
extensively studied class of \emph{resource-allocation problems}, in
which a fixed divisible budget is distributed among activities with
concave or convex returns; \citet{ibaraki1988resource} survey the
algorithmic theory, and \citet{patriksson2008survey} the continuous
nonlinear case. Neither the simplex constraint nor the separability of
the objective is new. What the present treatment adds is the
efficiency-function reduction and the elasticity criterion it yields,
which collapse the allocation question to a one-dimensional problem and
expose the aggregation exponent \(p\) as the parameter that governs the
transition between diversification and specialization. The resulting
comparison between the equalizing and concentrating optima may also be
read through majorization and Schur-convexity
\citep{marshall2011majorization}, which order symmetric separable
objectives by the evenness of the allocation; we return to this
connection in \Cref{sec:discussion} without developing it.
 
The remainder of the paper analyzes \eqref{eq:general-problem} under
elementary but progressively more revealing assumptions, beginning with
the crudest saturation model and ending with the efficiency-function
reduction and its consequences.
 
\section{Saturation and the Limits of Specialization}
\label{sec:saturation}

We first dispose of the crudest objection to a literal reading of the
Bruce Lee aphorism.  Suppose a skill improves linearly until mastery
and not at all thereafter.  For a mastery threshold \(\tau>0\), define
\begin{equation}
    f_\tau(x)
    =
    \min\Bigl\{\tfrac{x}{\tau},\,1\Bigr\},
    \label{eq:hard-saturation}
\end{equation}
and take all \(n\) skills identical with additive effectiveness,
\begin{equation}
    U(\vx)
    =
    \sum_{i=1}^n f_\tau(x_i).
    \label{eq:additive-hard}
\end{equation}

\begin{proposition}[No Practice Beyond Saturation]
\label{prop:no-overshoot}
Suppose \(n\geq2\) and let \(f_\tau\) be given by
\eqref{eq:hard-saturation}. If an allocation \(\vx\in\simplex_B\) satisfies
\(x_i>\tau\) and \(x_j<\tau\) for some \(i\neq j\), then \(\vx\) is not
optimal for \eqref{eq:additive-hard}.
\end{proposition}

\begin{proof}
Choose \(0<\varepsilon\leq\min\{x_i-\tau,\,\tau-x_j\}\) and transfer
\(\varepsilon\) units of practice from skill \(i\) to skill \(j\). The value
\(f_\tau(x_i)\) is unchanged, since \(x_i-\varepsilon\geq\tau\), while
\(f_\tau(x_j)\) increases by \(\varepsilon/\tau\). Total effectiveness
strictly increases.
\end{proof}

\begin{corollary}
\label{cor:hard-value}
For the hard-saturation model,
\[
    \max_{\vx\in\simplex_B}
    \sum_{i=1}^n f_\tau(x_i)
    =
    \min\Bigl\{\tfrac{B}{\tau},\,n\Bigr\}.
\]
If \(B=k\tau\) for an integer \(k\leq n\), then there is an optimal
allocation in which exactly \(k\) skills are fully mastered.
\end{corollary}

\begin{proof}
Since \(f_\tau(x)\leq x/\tau\) and \(f_\tau(x)\leq1\), every feasible
allocation obeys \(U(\vx)\leq\min\{B/\tau,n\}\). If \(B\leq n\tau\),
any allocation with all \(x_i\leq\tau\) and \(\sum_i x_i=B\) attains
\(\sum_i x_i/\tau=B/\tau\); if \(B>n\tau\), set every \(x_i=\tau\) and
distribute the surplus \(B-n\tau\) arbitrarily among these already
saturated skills, so that each obeys \(x_i\geq\tau\) and hence
\(f_\tau(x_i)=1\), giving the value \(n\). In either case the
allocation is feasible and meets the bound.
\end{proof}

Thus, with a budget \(B=10\,000\) and enough skills available, the
model admits an optimal allocation with one mastered skill when
\(\tau=10\,000\), with two when \(\tau=5\,000\), and with ten when
\(\tau=1\,000\); it does not, however, single such an allocation out,
since for \(B=k\tau\) every schedule that wastes no practice is
equally optimal.  This already refutes any reading of the aphorism as
a universal prescription: beyond saturation, further repetition of one
kick is wasted whenever a second useful kick remains unlearned.

Two features of this model must be removed to make the
breadth-versus-depth question sharp.  First, the optimum is far from
unique: when \(B<n\tau\), every allocation with \(x_i\leq\tau\) has
value \(B/\tau\), so the model separates wasteful from useful practice
but does not distinguish specialization from diversification.  Second,
the model has no curvature: practice is either fully productive or
fully wasted.  Both defects are addressed by introducing a genuine
learning curve and a genuine aggregation rule, to which we now turn.

\section{Concave Learning and Marginal Returns}
\label{sec:additive}

Assume now that the learning curves are increasing and strictly
concave, in keeping with the long-observed tendency for gains from
practice to diminish \citep{newell1981practice}, and consider the
weighted additive objective
\begin{equation}
    U(\vx)
    =
    \sum_{i=1}^n w_i\,f_i(x_i),
    \qquad w_i>0.
    \label{eq:weighted-additive}
\end{equation}

\begin{theorem}[Equalization of Marginal Returns]
\label{thm:marginal-equalization}
Suppose each \(f_i\) is continuously differentiable, strictly
increasing, and strictly concave on \(\Rplus\).  Then
\eqref{eq:weighted-additive} has a unique maximizer
\(\vx^*\in\simplex_B\), and there exists \(\lambda>0\) with
\[
    w_i\,f_i'(x_i^*)=\lambda
    \quad\text{whenever } x_i^*>0,
    \qquad
    w_i\,f_i'(0)\leq\lambda
    \quad\text{whenever } x_i^*=0.
\]
\end{theorem}

\begin{proof}
The objective is strictly concave and \(\simplex_B\) is compact and
convex, so a unique maximizer exists. The displayed relations are the
Karush--Kuhn--Tucker conditions for maximization over \(\simplex_B\);
see \citet[Ch.~5]{boyd2004convex}.
\end{proof}

\begin{corollary}[Equal Practice Under Additive Effectiveness]
\label{cor:equal-practice}
If \(f_1=\cdots=f_n=f\) with \(f\) strictly increasing and strictly
concave, and \(w_1=\cdots=w_n\), then the unique optimum is
\(x_1^*=\cdots=x_n^*=B/n\).
\end{corollary}

\begin{proof}
This follows directly from strict concavity, without any
differentiability assumption.  By Jensen's inequality
\citep[\S3.1]{hardy1952inequalities}, every \(\vx\in\simplex_B\)
satisfies
\[
    \frac1n\sum_{i=1}^n f(x_i)
    \leq
    f\!\Bigl(\frac1n\sum_{i=1}^n x_i\Bigr)
    =
    f\!\Bigl(\frac{B}{n}\Bigr),
\]
with equality, by strict concavity, if and only if \(x_1=\cdots=x_n\);
the constraint \(\sum_i x_i=B\) then forces \(x_i=B/n\). Hence
\(\sum_i f(x_i)\leq n\,f(B/n)\), attained uniquely at the balanced
allocation \(x_i^*=B/n\).
\end{proof}

\Cref{cor:equal-practice} points in precisely the opposite direction
from the aphorism: identical skills with concave returns call for
equal practice across the entire repertoire.  The discrepancy is not a
paradox---the additive objective \eqref{eq:weighted-additive} assumes
that every unit of acquired skill contributes independently to
effectiveness.  The next two sections replace that assumption and, in
doing so, restore the possibility that concentration is optimal.

\section{Aggregation: From Total Skill to a Single Best Technique}
\label{sec:aggregation}

The same learning curve can produce entirely different optimal
allocations under different notions of effectiveness.  Throughout this
section all skills share one increasing learning curve \(f\), and we
write \(\vs=(f(x_1),\ldots,f(x_n))\) for the vector of attained
skills.

At one extreme, total skill \(A_1(\vs)=\sum_i s_i\) returns us to the
additive objective, for which \Cref{cor:equal-practice} prescribes
equal practice. At the other extreme, effectiveness may be governed
entirely by the practitioner's strongest technique,
\[
    A_\infty(\vs)=\max_{1\leq i\leq n} s_i.
\]

\begin{proposition}[Complete Specialization]
\label{prop:max-specialization}
If \(f\) is increasing, then the allocation \((B,0,\ldots,0)\), up to
permutation, maximizes \(\max_i f(x_i)\) over \(\simplex_B\). If \(f\)
is strictly increasing, these are the only optima.
\end{proposition}

\begin{proof}
For every feasible \(\vx\) one has \(\max_i x_i\leq B\), so, \(f\)
being increasing, \(\max_i f(x_i)=f(\max_i x_i)\leq f(B)\), with
equality when a single skill receives the entire budget.
\end{proof}

The learning curve has not changed between
\Cref{cor:equal-practice,prop:max-specialization}; only the valuation
of acquired skill has. To interpolate, take for \(p\geq1\) the family
\begin{equation}
    U_p(\vx)
    =
    \Bigl(\sum_{i=1}^n f(x_i)^p\Bigr)^{1/p},
    \label{eq:lp-objective}
\end{equation}
which recovers total skill at \(p=1\) and, as \(p\to\infty\), the
maximum \(A_\infty\). We interpret the parameter \(p\) as measuring
the degree to which effectiveness rewards excellence in a few
techniques over broad competence.  Because the outer root is
increasing, maximizing \eqref{eq:lp-objective} is equivalent to
maximizing \(\sum_i g_p(x_i)\), where
\begin{equation}
    g_p(x):=f(x)^p.
    \label{eq:gp-def}
\end{equation}

The curvature of \(g_p\), not of \(f\) alone, drives the answer.

\begin{theorem}[Curvature Principle]
\label{thm:curvature-principle}
Let \(f:\Rplus\to\Rplus\) be increasing and let \(g_p\) be as in
\eqref{eq:gp-def}.
\begin{enumerate}[label=\textup{(\roman*)}]
\item If \(g_p\) is strictly concave on \([0,B]\), then the unique
maximizer of \(\sum_i g_p(x_i)\) over \(\simplex_B\) is
\(x_i^*=B/n\) for all \(i\).
\item If \(g_p\) is strictly convex on \([0,B]\) and \(g_p(0)=0\),
  then every maximizer is completely specialized:
  \(\vx^*=(B,0,\ldots,0)\), up to permutation.
\end{enumerate}
\end{theorem}

\begin{proof}
For (i), Jensen's inequality \citep[\S3.1]{hardy1952inequalities}
gives \(\frac1n\sum_i g_p(x_i)\leq g_p\bigl(\frac1n\sum_i
x_i\bigr)=g_p(B/n)\), with equality only when all coordinates
coincide. For (ii), a convex function vanishing at the origin is
superadditive: for \(x_i,x_j>0\), write each of \(x_i\) and \(x_j\) as
a convex combination of \(0\) and \(x_i+x_j\), namely
\(x_i=t(x_i+x_j)\) and \(x_j=(1-t)(x_i+x_j)\) with
\(t=x_i/(x_i+x_j)\); convexity and \(g_p(0)=0\) then give
\(g_p(x_i)\leq t\,g_p(x_i+x_j)\) and
\(g_p(x_j)\leq(1-t)g_p(x_i+x_j)\), whose sum is
\(g_p(x_i)+g_p(x_j)\leq g_p(x_i+x_j)\), strictly under strict
convexity.  Merging two positive coordinates therefore strictly
increases the objective, and iterating leaves a single positive
coordinate.
\end{proof}

\begin{example}[Power-Law Learning]
\label{ex:power-law}
Let \(f(x)=x^\alpha\) with \(0<\alpha<1\), so that \(g_p(x)=x^{\alpha p}\).
Then \(g_p\) is strictly concave for \(\alpha p<1\) and strictly convex for
\(\alpha p>1\), while at \(\alpha p=1\) one has
\(\sum_i f(x_i)^{p}=\sum_i x_i=B\) for every feasible allocation.
\end{example}

\begin{corollary}[Specialization Threshold]
\label{cor:threshold}
For \(f(x)=x^\alpha\), \(0<\alpha<1\), and the objective
\eqref{eq:lp-objective}, the optimum changes character at
\(p_c=1/\alpha\): below \(p_c\) equal practice is uniquely optimal, above
\(p_c\) complete specialization is optimal, and at \(p_c\) every feasible
allocation is optimal.
\end{corollary}

\begin{proof}
Immediate from \Cref{thm:curvature-principle,ex:power-law}.
\end{proof}

The power-law curve is instructive because it is degenerate: it never
yields an intermediate optimum, jumping directly from full breadth to
full specialization. To locate genuinely intermediate optima we must
look past curvature to the finer invariant of the next section.

\section[The Efficiency Function and the Optimal Repertoire]
{The Efficiency Function and the Optimal\\ Repertoire}
\label{sec:efficiency}

We now give the main reduction, replacing the \(n\)-variable problem
\(\max_{\simplex_B}\sum_i g_p(x_i)\) by the maximization of one scalar
function of a single real variable, and it identifies exactly when an
optimal schedule must be \emph{balanced}.

Throughout, \(f:\Rplus\to\Rplus\) is continuous with \(f(0)=0\), and
\(g_p=f^p\) as in \eqref{eq:gp-def}. Since \(g_p(0)=0\), a coordinate
that receives no practice contributes nothing, so only the practiced
skills matter.

\subsection{The Efficiency Reduction}
\label{subsec:efficiency}

\begin{definition}[Efficiency Function]
\label{def:efficiency}
The \emph{efficiency} of a practice level \(x>0\) is
\begin{equation}
    \eff(x)
    :=
    \frac{g_p(x)}{x}
    =
    \frac{f(x)^p}{x}.
    \label{eq:efficiency}
\end{equation}
\end{definition}

The efficiency measures effectiveness produced per unit of practice at
level \(x\). Its supremum controls the whole problem.

\begin{theorem}[Efficiency Bound]
\label{thm:efficiency-bound}
For every \(\vx\in\simplex_B\),
\[
    \sum_{i=1}^n g_p(x_i)
    \leq
    B\,\sup_{0<x\leq B}\eff(x).
\]
\end{theorem}

\begin{proof}
Let \(S=\supp\vx\). Coordinates outside \(S\) contribute
\(g_p(0)=0\), and for \(i\in S\) we have \(g_p(x_i)=x_i\,\eff(x_i)\).
Hence
\[
    \sum_{i=1}^n g_p(x_i)
    =
    \sum_{i\in S} x_i\,\eff(x_i)
    \leq
    \Bigl(\sup_{0<x\leq B}\eff(x)\Bigr)\sum_{i\in S} x_i
    =
    B\,\sup_{0<x\leq B}\eff(x),
\]
where \(x_i\leq B\) for each \(i\) justifies restricting the supremum to
\((0,B]\).
\end{proof}

The bound is attained precisely by an equal division of the budget, when
the arithmetic permits it.

\begin{definition}[Balanced Allocation]
\label{def:balanced}
For an integer \(k\) with \(1\leq k\leq n\), the \emph{balanced allocation
of order \(k\)} assigns \(B/k\) to each of \(k\) skills and \(0\) to the
rest.
\end{definition}

\begin{theorem}[Structural Reduction]
\label{thm:structural}
Suppose \(\eff\) attains its supremum over \((0,B]\) at a unique point
\(x^*\), and that \(k^*:=B/x^*\) is an integer with \(1\leq k^*\leq n\).
Then every maximizer of \(\max_{\vx\in\simplex_B}\sum_i g_p(x_i)\) is a
balanced allocation of order \(k^*\); that is, exactly \(k^*\) coordinates
equal \(B/k^*\) and the rest vanish. The maximum value is
\[
    B\,\eff(x^*)
    =
    k^*\,g_p\!\Bigl(\frac{B}{k^*}\Bigr).
\]
\end{theorem}

\begin{proof}
The balanced allocation of order \(k^*\) has value
\(k^* g_p(x^*)=k^* x^*\eff(x^*)=B\,\eff(x^*)\), which meets the bound of
\Cref{thm:efficiency-bound}; hence it is optimal, and the maximum value is
as stated. Conversely, let \(\vx\) be any maximizer, with support \(S\).
Equality in the proof of \Cref{thm:efficiency-bound} forces
\(\eff(x_i)=\eff(x^*)\) for every \(i\in S\). Since \(x^*\) is the unique
maximizer of \(\eff\), this gives \(x_i=x^*\) for all \(i\in S\), and then
\(|S|\,x^*=\sum_{i\in S}x_i=B\) forces \(|S|=B/x^*=k^*\).
\end{proof}

\Cref{thm:structural} is the rigorous form of the ``how many kicks?''
reduction. Whenever the efficiency function has a well-defined optimal
practice level \(x^*\) that divides the budget evenly, the optimal schedule
is unambiguous: practice exactly \(k^*=B/x^*\) techniques, each to depth
\(x^*\). The multivariable problem collapses to the search for a single
number.

It is convenient to record the reduction in the coordinates of the original
aphorism. For a balanced allocation of order \(k\), the objective
\eqref{eq:lp-objective} equals
\begin{equation}
    V_p(k)
    :=
    k^{1/p}\,f\!\Bigl(\frac{B}{k}\Bigr),
    \label{eq:Vk}
\end{equation}
and a direct computation gives the identity
\begin{equation}
    V_p(k)^p
    =
    k\,g_p\!\Bigl(\frac{B}{k}\Bigr)
    =
    B\,\eff\!\Bigl(\frac{B}{k}\Bigr).
    \label{eq:Vk-efficiency}
\end{equation}

\begin{definition}[Optimal Balanced Repertoire Size]
\label{def:repertoire}
The \emph{optimal balanced repertoire size} is
\(
    k^*\in\argmax_{1\leq k\leq n} V_p(k),
\)
the number of equally practiced skills that maximizes the objective among
balanced allocations. Outside the hypotheses of \Cref{thm:structural} it
need not equal the support size of a global optimum of the full problem.
\end{definition}

By \eqref{eq:Vk-efficiency}, maximizing \(V_p\) over the integers
\(1,\ldots,n\) is the same as maximizing the efficiency \(\eff\) over the
grid \(\{B/k:1\leq k\leq n\}\), and \Cref{thm:structural} states that when
the unconstrained efficiency optimum \(x^*\) lands on this grid, it is the
true optimum of the full problem. When \(x^*\) does not land on the grid,
\Cref{thm:efficiency-bound} still furnishes the upper bound
\(B\,\eff(x^*)\), and the best balanced schedule is found by maximizing
\(\eff(B/k)\) over \(k\in\{1,\ldots,n\}\). When \(\eff\) is unimodal, as it
is under the hypotheses of \Cref{cor:elasticity-structural}, this maximizer
is the better of the two grid points bracketing \(x^*\); in general the grid
maximizer need not bracket \(x^*\), and the resulting shortfall from
\(B\,\eff(x^*)\) is the price of integrality.

\subsection{The Elasticity Criterion}
\label{subsec:elasticity}

To find \(x^*\) we differentiate the efficiency. Assume \(f\) is
differentiable and positive on \((0,B]\). Writing \(\log\eff(x)=p\log
f(x)-\log x\), we obtain
\begin{equation}
    \bigl(\log\eff\bigr)'(x)
    =
    p\,\frac{f'(x)}{f(x)}-\frac1x
    =
    \frac1x\bigl(p\,\elas_f(x)-1\bigr),
    \label{eq:logderiv}
\end{equation}
where
\begin{equation}
    \elas_f(x)
    :=
    \frac{x\,f'(x)}{f(x)}
    \label{eq:elasticity}
\end{equation}
is the \emph{elasticity} of the learning curve, the local percentage gain
in skill per percentage gain in practice. Setting \eqref{eq:logderiv} to
zero yields the stationarity condition
\begin{equation}
    \elas_f(x^*)=\frac1p.
    \label{eq:elasticity-condition}
\end{equation}

\begin{proposition}[Elasticity Criterion]
\label{prop:elasticity}
An interior maximizer \(x^*\in(0,B)\) of the efficiency function satisfies
\(\elas_f(x^*)=1/p\). Moreover \(\elas_f(x)>1/p\) means that efficiency is
locally increasing in \(x\), so that deeper practice of fewer skills is
locally preferable, whereas \(\elas_f(x)<1/p\) means that efficiency is
locally decreasing, so that broader practice of more skills is locally
preferable.
\end{proposition}

\begin{proof}
Stationarity is \eqref{eq:elasticity-condition}. The monotonicity
statements read off the sign of \(p\,\elas_f(x)-1\) in
\eqref{eq:logderiv}, since deeper practice per skill corresponds to larger
\(x=B/k\) and hence smaller \(k\).
\end{proof}

Equation \eqref{eq:elasticity-condition} is the compact answer to
``how many kicks?'': the optimal breadth of practice is fixed by the
balance of two dimensionless quantities, the elasticity of learning
and the reciprocal aggregation parameter \(1/p\). The following
consequence isolates the case in which the balance is decisive.

\begin{corollary}[Unimodal Efficiency]
\label{cor:elasticity-structural}
Suppose \(\elas_f\) is strictly decreasing on \((0,B]\) and satisfies
\(\elas_f(x^*)=1/p\) at some \(x^*\in(0,B]\), with \(k^*:=B/x^*\) an integer
in \(\{1,\ldots,n\}\). Then \(x^*\) is the unique maximizer of \(\eff\) over
\((0,B]\), and every optimum of \(\max_{\vx\in\simplex_B}\sum_i g_p(x_i)\) is
a balanced allocation of order \(k^*\).
\end{corollary}

\begin{proof}
By \eqref{eq:logderiv}, \((\log\eff)'\) has the sign of
\(p\,\elas_f(x)-1\). Since \(\elas_f\) strictly decreases and equals \(1/p\)
at \(x^*\), this sign is positive for \(x<x^*\) and, when \(x^*<B\), negative
for \(x^*<x\leq B\); hence \(\eff\) strictly increases on \((0,x^*]\) and
strictly decreases on \([x^*,B]\) (the latter interval being degenerate when
\(x^*=B\)), so \(x^*\) is its unique maximizer. \Cref{thm:structural} then
applies.
\end{proof}

We revisit the two running examples in this light.

\begin{example}[Power-Law Learning, Revisited]
\label{ex:power-law-elasticity}
For \(f(x)=x^\alpha\) the elasticity is the constant \(\elas_f(x)=\alpha\),
so \eqref{eq:elasticity-condition} has a solution only in the degenerate
case \(\alpha=1/p\). Correspondingly \(\eff(x)=x^{\alpha p-1}\) is
monotone: increasing for \(\alpha p>1\), forcing \(x^*=B\) and \(k^*=1\)
(specialize); decreasing for \(\alpha p<1\), forcing \(x^*\to0\) and hence
the largest admissible \(k\), namely \(k^*=n\) (diversify); and constant for
\(\alpha p=1\). This recovers \Cref{cor:threshold} from the efficiency
viewpoint and explains why constant elasticity admits no interior optimum.
\end{example}

\begin{example}[Exponential Saturation]
\label{ex:exponential}
Let \(f(x)=1-e^{-\beta x}\) with \(\beta>0\), a learning curve that rises
linearly near the origin and saturates at \(1\). Its elasticity is
\[
    \elas_f(x)
    =
    \frac{\beta x\,e^{-\beta x}}{1-e^{-\beta x}}
    =
    \frac{\beta x}{e^{\beta x}-1},
\]
which decreases strictly from \(1\) at \(x=0^+\) to \(0\) as
\(x\to\infty\). Hence for every \(p>1\) the equation \(\elas_f(x)=1/p\) has a
unique positive solution \(x_p\), and, since \(\elas_f\) is strictly
decreasing, the efficiency \(\eff\) is strictly unimodal with unique
maximizer \(x_p\).

Three notions of the optimal repertoire size must here be distinguished.
The \emph{continuous relaxation} treats the repertoire size as a real
variable \(k=B/x\) and maximizes \(\eff(B/k)\) over \(k\in[1,n]\). Since
\(\eff\) is strictly unimodal in \(x\) with peak at \(x_p\), the map
\(k\mapsto\eff(B/k)\) is strictly unimodal in \(k\) with peak at the
unconstrained optimum \(k_{\mathrm c}=B/x_p\); the constrained maximizer
over \([1,n]\) is therefore the projection of \(k_{\mathrm c}\) onto
\([1,n]\), equal to \(k_{\mathrm c}\) itself when \(B/n\leq x_p\leq B\) and
to the nearer endpoint otherwise. The \emph{optimal balanced repertoire
size} of \Cref{def:repertoire} maximizes \(\eff(B/k)\) over the integers
\(k\in\{1,\ldots,n\}\); by the same unimodality it is \(\lfloor
k_{\mathrm c}\rfloor\) or \(\lceil k_{\mathrm c}\rceil\) when
\(k_{\mathrm c}\in[1,n]\), and the nearer endpoint \(1\) or \(n\)
otherwise. The \emph{unrestricted problem}
\(\max_{\vx\in\simplex_B}\sum_i g_p(x_i)\) coincides with the balanced
optimum when \(B/x_p\) is an integer in \(\{1,\ldots,n\}\): then
\Cref{cor:elasticity-structural} applies and every global optimum is a
balanced allocation of order \(B/x_p\). When \(B/x_p\) is not an
integer, however, the bound \(B\,\eff(x_p)\) of \Cref{thm:efficiency-bound}
cannot be attained, since attainment would require every practiced skill to
receive exactly \(x_p\) and hence \(B/x_p\) to be an integer; the
unrestricted optimum is then strictly smaller than the relaxation predicts,
and the present results locate the best \emph{balanced} schedule without
establishing that it solves the unrestricted problem.

This is nonetheless the simplest model in which the optimal balanced
repertoire is genuinely intermediate, answering ``one kick or ten
thousand?'' with neither extreme: the saturating curve has strictly
decreasing elasticity, so a definite interior repertoire is optimal among
balanced schedules, and optimal over all schedules when \(B/x_p\) is an
integer.
\end{example}

The contrast between \Cref{ex:power-law-elasticity,ex:exponential} is the
heart of the matter. A learning curve of constant elasticity gives no
interior balance and forces one of the two extremes of the aphorism; a
genuinely saturating curve gives strictly decreasing elasticity, an
interior solution of \eqref{eq:elasticity-condition}, and a well-defined
finite number of kicks worth mastering.

\section{Skills of Unequal Difficulty}
\label{sec:heterogeneous}

The preceding sections assumed identical skills. We now permit different
learning rates while retaining a common curve shape, writing
\begin{equation}
    f_i(x)=f\!\Bigl(\frac{x}{\tau_i}\Bigr),
    \label{eq:hetero-curve}
\end{equation}
where \(\tau_i>0\) is a characteristic practice scale for skill \(i\); a
small \(\tau_i\) marks a skill acquired quickly.

Under additive effectiveness with weights \(w_i>0\), the problem is to
maximize \(\sum_i w_i\,f(x_i/\tau_i)\) over \(\simplex_B\), and
\Cref{thm:marginal-equalization} gives, at an interior optimum,
\begin{equation}
    \frac{w_i}{\tau_i}\,
    f'\!\Bigl(\frac{x_i^*}{\tau_i}\Bigr)
    =\lambda
    \qquad(i\in\supp\vx^*).
    \label{eq:hetero-foc}
\end{equation}
Practice is therefore \emph{not} distributed in proportion to learning
difficulty. It is distributed so that the weighted marginal return per unit
of practice is equalized across the practiced skills, which may favor easy
or hard skills depending on how the weight \(w_i\) and the scale \(\tau_i\)
interact in \eqref{eq:hetero-foc}.

The hard-mastery version has a different, and combinatorially familiar,
structure. Suppose skill \(i\) is worth \(w_i\) once at least \(\tau_i\)
units of practice are invested and worth nothing otherwise. Write
\(y_i=\mathbf{1}\{x_i\geq\tau_i\}\) for the mastery indicator. Under the
standing constraint \(\sum_i x_i=B\), an optimal schedule may still be
forced to assign practice that produces no value, either below a threshold
or as unavoidable excess when \(B>\sum_i\tau_i\); but such subthreshold and
excess allocations do not affect the total value \(\sum_i w_i y_i\). The
optimal value is therefore obtained by choosing which skills to master,
subject to their thresholds summing to at most \(B\), and the surplus budget
is distributed harmlessly. This is the \(0\)--\(1\) knapsack problem
\begin{equation}
    \max\ \sum_{i=1}^n w_i y_i
    \quad\text{subject to}\quad
    \sum_{i=1}^n \tau_i y_i\leq B,
    \quad y_i\in\{0,1\}.
    \label{eq:knapsack}
\end{equation}

\begin{proposition}[Knapsack Reduction]
\label{prop:knapsack}
The optimal value of the heterogeneous hard-mastery model equals that of the
\(0\)--\(1\) knapsack problem \eqref{eq:knapsack} with item values \(w_i\)
and weights \(\tau_i\).
\end{proposition}

In general \eqref{eq:knapsack} is \NP-hard \citep{martello1990knapsack},
so heterogeneity in value can make the ``how many kicks?'' question
genuinely difficult. When the skills are equally valuable, however, the
answer is elementary and greedy.

\begin{proposition}[Greedy Mastery for Equal Values]
\label{prop:greedy-threshold}
Suppose every mastered skill is worth \(1\), and order the thresholds as
\(\tau_{(1)}\leq\cdots\leq\tau_{(n)}\). Then the maximum number of mastered
skills is the largest integer \(k\) with
\(\sum_{j=1}^k \tau_{(j)}\leq B\), achieved by mastering the \(k\) cheapest
skills.
\end{proposition}

\begin{proof}
Mastering any \(k\) skills costs at least \(\sum_{j=1}^{k}\tau_{(j)}\), the
sum of the \(k\) smallest thresholds; if this exceeds \(B\), no \(k\) skills
can be mastered, and if it does not, the \(k\) cheapest can be. The maximum
feasible \(k\) is therefore as stated.
\end{proof}

\section{Discussion}
\label{sec:discussion}

The Bruce Lee aphorism contrasts two extreme practice schedules: the fully
specialized allocation \((B,0,\ldots,0)\) and, in effect, the fully
diversified one \((B/n,\ldots,B/n)\).  Our analysis shows that
neither extreme is universally correct and, more usefully, identifies
what the correct answer depends on.

Three conclusions organize the analysis. First, saturation caps the value
of specialization: once further repetition yields no further skill, it is
wasted whenever a fresh skill remains, as \Cref{cor:hard-value} makes
quantitative. Second, diminishing returns do not by themselves settle the
question, because the same concave learning curve favors breadth under an
additive objective (\Cref{cor:equal-practice}) and depth under a
strongest-skill objective (\Cref{prop:max-specialization}); the outcome
turns on the aggregation rule as much as on the learning curve. Third, and
most sharply, the entire trade-off is captured by a single scalar
efficiency function \(\eff(x)=f(x)^p/x\). Under
\Cref{thm:structural} its optimizer determines a balanced schedule, and by
the elasticity criterion \eqref{eq:elasticity-condition} the optimal number
of kicks is fixed by the equation \(\elas_f(x^*)=1/p\): the learning
elasticity on one side, the reciprocal of the aggregation exponent on the
other.

Seen this way, Bruce Lee's saying about ten thousand kicks is not a
universal theorem about skill learning through practice, but rests on
an implicit pair of assumptions---one about the shape of learning and
one about the structure of performance. The man who practices one kick
ten thousand times is vindicated when his learning curve keeps
elasticity above \(1/p\) all the way to \(x=B\), for the efficiency
then increases throughout and complete specialization is globally
optimal; the man who practices ten thousand kicks once is favored when
the elasticity stays below \(1/p\) throughout, for the efficiency then
decreases and, among balanced schedules, the broadest is best.  The two
halves are not quite symmetric: increasing efficiency makes
specialization optimal over all schedules, whereas decreasing
efficiency establishes breadth only among balanced schedules, in
keeping with the balanced-versus-unrestricted distinction of
\Cref{def:repertoire,ex:exponential}; global optimality of the fully
diversified allocation would require a further condition, such as
concavity of \(g_p\). These are in any case sufficient conditions
rather than a complete characterization: between them lies the
interesting case, and there the number of kicks worth practicing is
neither one nor ten thousand, but an intermediate value governed by
\(B/x^*\).

Several natural extensions lie beyond the present framework and would
require genuine modification of it. Skill interactions that violate
separability can dissolve the scalar efficiency reduction altogether, since
the objective is then no longer a sum of one-variable terms; and budgets
that grow with acquired competence alter the feasible set in ways the
current analysis does not address. We regard these as directions for further
work rather than as immediate consequences of the reduction.

The equalization and concentration extremes established above
(\Cref{cor:equal-practice,prop:max-specialization}) may also be phrased in
the language of majorization \citep{marshall2011majorization}, which orders
symmetric separable objectives by how evenly the budget is spread; we do not
develop that reformulation here.

We close by returning to the observation with which the paper began.
Bruce Lee's preference for depth is not overturned by the analysis but
located by it: the elasticity criterion \eqref{eq:elasticity-condition}
marks exactly where practicing one kick ten thousand times gives way,
through a regime of definite intermediate breadth, to practicing many
kicks fewer times each. The saying is thereby not merely illustrated but
sharpened, its force and its limits fixed by the assumptions it silently
makes about how skill is learned and how performance is judged.



\end{document}